\documentclass[a4paper]{amsart}

\usepackage{amsmath,amssymb,amsthm}
\usepackage[utf8]{inputenc}
\usepackage[T1]{fontenc}
\usepackage{libertine}
\usepackage[libertine,cmintegrals,cmbraces]{newtxmath}

\usepackage{microtype}
\usepackage{xcolor}
\definecolor{darkred}{RGB}{139,0,0}
\definecolor{darkblue}{RGB}{0,0,139}
\definecolor{darkgreen}{RGB}{0,70,0}
\usepackage{hyperref}
\hypersetup{
  colorlinks   = true, 
  urlcolor     = darkred, 
  linkcolor    = darkred, 
  citecolor   = darkgreen}
\usepackage{tikz}
\usepackage{tikz-cd}
\usepackage[capitalise]{cleveref}
\usepackage{scalerel}
\usepackage{mathtools}
\usepackage{enumitem}
\setenumerate[1]{leftmargin=*,labelindent=0pt,label=(\roman*),}

\newtheorem{bigthm}{Theorem}
\newtheorem{bigcor}[bigthm]{Corollary}

\newtheorem{thm}{Theorem}[section]

\newtheorem{lem}[thm]{Lemma}
\newtheorem{prop}[thm]{Proposition}
\newtheorem{cor}[thm]{Corollary}

\theoremstyle{definition}

\theoremstyle{remark}

\newtheorem{rem}[thm]{Remark}
\newtheorem*{nrem}{Remark}

\newcommand\dslash{/\mkern-6mu/}
\newcommand{\Diff}{\ensuremath{\operatorname{Diff}^{\scaleobj{0.8}{+}}}}
\newcommand{\Diffuo}{\ensuremath{\operatorname{Diff}}}
\newcommand{\BDiff}{\ensuremath{\operatorname{BDiff}^{\scaleobj{0.8}{+}}}}
\newcommand{\BDiffuo}{\ensuremath{\operatorname{BDiff}}}
\newcommand{\Fr}{\ensuremath{\operatorname{Fr}^{\scaleobj{0.8}{+}}}}
\newcommand{\oH}{\ensuremath{\operatorname{H}}}
\newcommand{\bfZ}{\ensuremath{\mathbf{Z}}}
\newcommand{\bfS}{\ensuremath{\mathbf{S}}}
\newcommand{\M}{\ensuremath{\mathbf{M}}}
\newcommand{\MSO}{\ensuremath{\mathbf{MSO}}}
\newcommand{\MO}{\ensuremath{\mathbf{MO}}}
\newcommand{\Mtheta}{\ensuremath{\mathbf{M\theta}}}
\newcommand{\MT}{\ensuremath{\mathbf{MT}}}
\newcommand{\BO}{\ensuremath{\operatorname{BO}}}
\newcommand{\BSO}{\ensuremath{\operatorname{BSO}}}
\newcommand{\SO}{\ensuremath{\operatorname{SO}}}
\newcommand{\EO}{\ensuremath{\operatorname{EO}}}
\newcommand{\Thom}{\ensuremath{\mathbf{Th}}}
\newcommand{\hAut}{\ensuremath{\operatorname{hAut}}}
\newcommand{\Stab}{\ensuremath{\operatorname{Stab}}}
\newcommand{\MTO}{\ensuremath{\mathbf{MTO}}}
\newcommand{\MTString}{\ensuremath{\mathbf{MTString}}}
\newcommand{\aut}{\ensuremath{\operatorname{Aut}}}
\newcommand{\im}{\ensuremath{\operatorname{im}}}
\newcommand{\MSpin}{\ensuremath{\mathbf{MSpin}}}
\newcommand{\MString}{\ensuremath{\mathbf{MString}}}
\newcommand{\tmf}{\ensuremath{\mathbf{tmf}}}
\newcommand{\Coker}{\ensuremath{\operatorname{coker}}}
\newcommand{\fr}{\ensuremath{\operatorname{fr}}}
\newcommand{\Mod}[1]{\ (\mathrm{mod}\ #1)}
\newcommand{\interior}[1]{\ensuremath{\operatorname{int}(#1)}}
\newcommand{\Kerv}{\ensuremath{\operatorname{Kerv}}}
\newcommand{\bP}{\ensuremath{\operatorname{bP}}}

\newcommand{\num}{\ensuremath{\operatorname{num}}}
\newcommand{\der}{\ensuremath{der}}
\newcommand{\ra}{\rightarrow}
\newcommand{\lra}{\longrightarrow}
\newcommand{\lla}{\longleftarrow}
\newcommand{\xlra}[1]{\overset{#1}{\longrightarrow}}
\newcommand{\mapnoname}[4]{\ensuremath{\begin{array}{rcl} 
      #1 & \longrightarrow & #2 \\[0.3em] 
      #3 & \longmapsto & #4
    \end{array}}}
    
\begin{document}

\title{On characteristic classes of exotic manifold bundles}
\author{Manuel Krannich}
\email{krannich@dpmms.cam.ac.uk}
\address{Centre for Mathematical Sciences, Wilberforce Road, Cambridge CB3 0WB, UK}
\begin{abstract}
Given a closed simply connected manifold $M$ of dimension $2n\ge6$, we compare the ring of characteristic classes of smooth oriented bundles with fibre $M$ to the analogous ring resulting from replacing $M$ by the connected sum $M\sharp\Sigma$ with an exotic sphere $\Sigma$. We show that, after inverting the order of $\Sigma$ in the group of homotopy spheres, the two rings in question are isomorphic in a range of degrees. Furthermore, we construct infinite families of examples witnessing that inverting the order of $\Sigma$ is necessary.
\end{abstract}

\maketitle

The classifying space $\BDiff(M)$ of the topological group of orientation-preserving diffeomorphisms of a closed oriented manifold $M$ in the smooth Whitney-topology classifies smooth oriented fibre bundles with fibre $M$. Its cohomology $\oH^*(\BDiff(M))$ is the ring of characteristic classes of such bundles and is thus of great interest from the point of view of geometric topology. Initiated by Madsen--Weiss' solution of the Mumford conjecture on the moduli space of Riemann surfaces \cite{MadsenWeiss}, there has been significant progress in the study of $\oH^*(\BDiff(M))$ in recent years, including in high dimensions. A programme of Galatius--Randal-Williams \cite{GRWstable,GRWII,GRWI} culminated in an identification of the cohomology in consideration in a range of degrees in purely homotopy theoretical terms for all simply connected manifolds $M$ of dimension $2n\ge6$. Analogous to the case of surfaces, this range depends on the \emph{genus} of $M$, defined as \[g(M)\coloneqq\max\{g\ge0\mid \text{there exists a manifold }N\text{ with }M\cong N\sharp(S^n\times S^n)^{\sharp g}\}.\]
This work is concerned with the behaviour of the cohomology $\oH^*(\BDiff(M))$ when changing the smooth structure of the underlying $d$-manifold $M$ on an embedded disc of codimension zero; that is, when replacing $M$ by the connected sum $M\sharp\Sigma$ with an exotic sphere $\Sigma\in\Theta_d$. Here $\Theta_d$ denotes the finite abelian group of oriented homotopy spheres, classically studied by Kervaire--Milnor \cite{KervaireMilnor}. In the first part, we use the work of Galatius--Randal-Williams to show that for closed simply connected manifolds $M$ of dimension $2n\ge 6$, the cohomology $\oH^*(\BDiff(M))$ is insensitive to replacing $M$ by $M\sharp\Sigma$ in a range of degrees, at least after inverting the order of $\Sigma$. As our methods are more of homological than of cohomological nature, we state our results in homology.

\begin{bigthm}\label{theorem:maintheorem}Let $M$ be a closed, oriented, simply connected manifold of dimension $2n\ge 6$ and $\Sigma\in\Theta_{2n}$ an exotic sphere. There is a zig-zag of maps of spaces inducing an isomorphism
\[\textstyle{\oH_*(\BDiff(M);\bfZ[\frac{1}{k}])\cong\oH_*(\BDiff(M\sharp\Sigma);\bfZ[\frac{1}{k}])}\] in degrees $*\le\frac{g(M)-3}{2}$, where $k$ denotes the order of $\Sigma$ in $\Theta_{2n}$.
\end{bigthm}

\begin{nrem}Using recent work of Friedrich \cite{Friedrich}, one can enhance Theorem~\ref{theorem:maintheorem} to all oriented, closed, connected manifolds $M$ of dimension $2n\ge6$ whose associated group ring $\bfZ[\pi_1M]$ has finite unitary stable rank. This applies for instance if the fundamental group is virtually polycyclic, so in particular if it is finite or finitely generated abelian.
\end{nrem}

In the second part of this work, we focus on the family of manifolds $W_g=\sharp^g(S^n\times S^n)$ to address the question of whether Theorem~\ref{theorem:maintheorem} fails without inverting the order of the homotopy sphere, i.e.\,with integral coefficients. To state our first result in that direction, we denote by $\MO\langle n\rangle$ the Thom-spectrum of the $n$-connected cover $\BO\langle n\rangle\ra\BO$. By the classical theorem of Pontryagin--Thom, the ring of homotopy groups $\pi_*\MO\langle n\rangle$ is isomorphic to the ring $\Omega_*^{\langle n\rangle}$ of bordism classes of closed manifolds equipped with a lift of their stable normal bundle along $\BO\langle n\rangle\ra\BO$. An oriented homotopy sphere $\Sigma$ of dimension $2n$ has, up to homotopy, a unique such lift that is compatible with its orientation and hence defines a canonical class $[\Sigma]\in\Omega^{\langle n\rangle}_{2n}$. As is common, we denote by $\eta\in\pi_1\bfS$ the generator of the first stable homotopy group of the sphere spectrum.

\begin{bigthm}\label{theorem:H1}For $g\ge0$ and $\Sigma\in\Theta_{2n}$ with $2n\ge6$, there is an exact sequence \[\bfZ/2\lra\oH_1(\BDiff(W_g);\bfZ)\lra\oH_1(\BDiff(W_g\sharp\Sigma);\bfZ)\lra 0.\] If the product $\eta\cdot[\Sigma]\in\pi_{2n+1}\MO\langle n\rangle$ does not vanish, then the first map is nontrivial. Moreover, the converse holds for $g\ge2$.
\end{bigthm}

Note that $\oH_1(\BDiff(W_g);\bfZ)$ agrees with the abelianisation of the group $\pi_0\Diff(W_g)$ of isotopy classes of diffeomorphisms of $M$. It follows from a result of Kreck \cite[Thm\,2]{Kreck} that this group is finitely generated for $2n\ge6$, so the previous theorem implies that $\oH_1(\BDiff(W_g);\bfZ)$ and $\oH_1(\BDiff(W_g\sharp\Sigma);\bf{Z})$ cannot be isomorphic if the product $\eta\cdot[\Sigma]\in\pi_{2n+1}\MO\langle n\rangle$ is nontrivial. From computations in stable homotopy theory, we derive the existence of infinite families of homotopy spheres for which this product does not vanish, hence for which the integral version of Theorem~\ref{theorem:maintheorem} fails in degree $1$. Such examples exist already in dimension 8---the first possible dimension. Combining Theorem~\ref{theorem:H1} with work of Kreck \cite{Kreck}, we also find $\Sigma$ such that $\pi_0\Diff(W_g)$ and $\pi_0\Diff(W_g\sharp\Sigma)$ are nonisomorphic, but become isomorphic after abelianisation.

\begin{bigcor}\label{corollary:H1}There are $\Sigma\in\Theta_{2n}$ in infinitely many dimensions $2n$ such that  
\[\oH_1(\BDiff(W_g);\bfZ)\quad\text{and}\quad\oH_1(\BDiff(W_g\sharp\Sigma);\bfZ)\] are not isomorphic for $g\ge0$. Furthermore, there are  $\Sigma\in\Theta_{8k+2}$ for all $k\ge1$ such that 
\[\pi_1\BDiff(W_g)\quad\text{and}\quad\pi_1\BDiff(W_g\sharp\Sigma)\] are not isomorphic for $g\ge0$, but have isomorphic abelianisations for $g\ge 2$.
\end{bigcor}

\begin{nrem}
The final parts of Theorem~\ref{theorem:H1} and Corollary~\ref{corollary:H1} still hold for $g=1$ in most cases, but often fail for $g=0$ (see Remark~\ref{remark:theoremholds}).
\end{nrem}

By Theorem~\ref{theorem:H1}, the first homology groups of $\BDiff(W_g)$ and $\BDiff(W_g\sharp\Sigma)$ are isomorphic after inverting $2$. More generally, this holds with $W_g$ replaced by any simply connected manifold, which raises the question of whether the failure for Theorem~\ref{theorem:maintheorem} to hold integrally is purely $2$-primary. We answer this question in the negative by proving the following.

\begin{bigthm}\label{theorem:H3}For $g\ge0$ and $\Sigma\in\Theta_{2n}$ with $2n\ge6$, there is an isomorphism
\[\textstyle{\oH_*(\BDiff(W_g);\bfZ[\frac{1}{2}])\cong\oH_*(\BDiff(W_g\sharp\Sigma);\bfZ[\frac{1}{2}])}\] in degrees $*\le 2$ and furthermore an exact sequence
\[\textstyle{\bfZ[\frac{1}{2}]\lra\oH_3(\BDiff(W_g);\bfZ[\frac{1}{2}])\lra\oH_3(\BDiff(W_g\sharp\Sigma);\bfZ[\frac{1}{2}])}\lra0.\] If $2n=10$ and the order of $\Sigma$ is divisible by $3$, then the first morphism is nontrivial.
\end{bigthm}

It follows from a recent result of Kupers \cite[Cor.\,C]{Kupers} that $\oH_3(\BDiff(W_g);\bfZ)$ is finitely generated for $2n\ge 6$, so $\oH_3(\BDiff(W_g);\bfZ[\frac{1}{2}])$ and $\oH_3(\BDiff(W_g\sharp\Sigma);\bfZ[\frac{1}{2}])$ cannot be isomorphic if the first map in the sequence of Theorem~\ref{theorem:H3} is nontrivial. As $\Theta_{10}\cong\bfZ/6$, this holds for four of the six homotopy $10$-spheres by the second part of Theorem~\ref{theorem:H3}.

\begin{bigcor}\label{corollary:H3}For $g\ge0$ and $\Sigma\in\Theta_{10}$ whose order is divisible by $3$, the groups 
\[\textstyle{\oH_3(\BDiff(W_g);\bfZ[\frac{1}{2}])\quad\text{and}\quad\oH_3(\BDiff(W_g\sharp\Sigma);\bfZ[\frac{1}{2}])}\] are not isomorphic.
\end{bigcor}

\clearpage

\begin{nrem}
\ 
\begin{enumerate}
\item As an application of parametrised smoothing theory, Dwyer--Szczarba \cite{DwyerSzczarba} compared the homotopy types of the unit components $\Diffuo_0(M)\subset\Diff(M)$ for different smooth structures of $M$. Their work in particular implies that for a closed manifold $M$ of dimension $d\ge5$ and a homotopy sphere $\Sigma\in\Theta_d$, the spaces $\BDiffuo_0(M)$ and $\BDiffuo_0(M\sharp\Sigma)$ have the same $\bfZ[\frac{1}{k}]$-homotopy type for $k$ the order of $\Sigma\in\Theta_{d}$, which results in an analogue of Theorem~\ref{theorem:maintheorem} for $\BDiffuo_0(M)$ without assumptions on the degree. However, in general the (co)homology of $\BDiff(M)$ and $\BDiffuo_0(M)$ is very different and we do not know whether Theorem~\ref{theorem:maintheorem} holds outside the stable range in the sense of of Galatius--Randal-Williams.
\item
Our methods also show the analogues of Corollary~\ref{corollary:H3} and the first part of Corollary~\ref{corollary:H1} for homotopy instead of homology groups which implies the existence of $\Sigma\in\Theta_{2n}$ for which $\pi_i\BDiff(W_g)$ and $\pi_i\BDiff(W_g\sharp \Sigma)$ are not isomorphic for some $i$ and all $g$. For $g=0$, such examples have been constructed by Habegger--Szczarba \cite{HabeggerSzczarba}.
\end{enumerate}
\end{nrem}

\subsection*{Acknowledgements}
I would like thank Oscar Randal-Williams for asking me a question which led to this work, as well as for many enlightening discussions and his hospitality at the University of Cambridge. Furthermore, I am grateful to Søren Galatius, Alexander Kupers, and Jens Reinhold for valuable comments, and to Mauricio Bustamante for making me aware of the work of Dwyer--Szczarba and Habegger--Szczarba. I was supported by the Danish National Research Foundation through the Centre for Symmetry and Deformation (DNRF92) and by the European Research Council (ERC) under the European Union’s Horizon 2020 research and innovation programme (grant agreement No 682922).

\section{Exotic spheres and parametrised Pontryagin\textendash Thom theory}
After a brief recollection on bordism theory and groups of homotopy spheres, we recall high-dimensional parametrised Pontryagin--Thom theory à la Galatius--Randal-Williams and prove Theorem~\ref{theorem:maintheorem}.

\subsection{Bordism theory}\label{section:bordism}Let $\theta\colon B\ra \BO$ be a fibration. A \emph{tangential}, respectively \emph{normal}, $\theta$-structure of a manifold $M$ is a lift $\ell_M\colon M\ra B$ of its stable tangent, respectively normal, bundle $M\ra \BO$ along $\theta$, up to homotopy over $\BO$. The collection of bordism classes of closed $d$-manifolds equipped with a normal $\theta$-structure forms an abelian group $\Omega_d^\theta$ under disjoint union (see e.g.\,\cite[Ch.\,2]{Stong} for details). By the classical Pontryagin--Thom theorem, this group is isomorphic to the $d$th homotopy group $\pi_d\Mtheta$ of the Thom spectrum $\Mtheta$ associated to $\theta$. Normal $\theta$-structures of a manifold are in natural bijection to tangential $\theta^\perp$-structures, where $\theta^\perp\colon B^\perp\ra \BO$ is the pullback of $\theta$ along the canonical involution $-1\colon\BO\ra\BO$, so we do not distinguish between them.

\subsection{Exotic spheres}\label{section:homotopyspheres}The initial observation of Kervaire--Milnor's \cite{KervaireMilnor} classification of homotopy spheres is that the collection $\Theta_d$ of h-cobordism classes of closed oriented $d$-manifolds with the homotopy type of a $d$-sphere forms an abelian group under taking connected sum. For $d\ge 5$, this group can be described equivalently as the group of oriented $d$-manifolds homeomorphic to the $d$-sphere, modulo orientation-preserving diffeomorphisms. Kervaire--Milnor established an exact sequence of the form
\[0\lra \bP_{d+1}\lra\Theta_{d}\lra\Coker(J)_d\xlra{\Kerv_d}\bfZ/2,\] where $\bP_{d+1}\subset\Theta_d$ is the subgroup of homotopy spheres bounding a parallelisable manifold and $\Coker(J)_d=\pi_d\bfS/\im(J)_d$ is the cokernel of the stable $J$-homomorphism $J\colon\pi_d\operatorname{O}\ra\pi_d\bfS$ to the stable homotopy groups of spheres. In particular, the group $\Theta_d$ is finite. Noting that $\bfS\simeq\mathbf{M}(\fr\colon\EO\ra\BO)$, the ring $\pi_*\bfS$ can be viewed equivalently as the ring $\Omega_*^{\fr}$ of bordism classes of stably framed manifolds. The morphism $\Theta_d\ra\Coker(J)_d$ is induced by assigning a homotopy sphere $\Sigma\in\Theta_d$ the class $[\Sigma,\ell]\in\Coker(J)_d$ of its underlying manifold together with a choice of a stable framing, using that homotopy spheres are stably parallelisable, and the morphism $\Kerv\colon\Coker(J)_d\ra\bfZ/2$ is the so-called \emph{Kervaire invariant}. Kervaire--Milnor showed that the subgroup $\bP_{d+1}\subset\Theta_d$ is cyclic and computed its order in most cases: for $d=4k-1$ and $k\ge2$ it is of order $2^{2k-2}(2^{2k-1}-1)\num(|4B_{2k}|/k)$, where $B_{2k}$ is the $2k$th Bernoulli number (this uses Adams'  \cite{Adams} computation of $\im(J)$ and Quillen's \cite{Quillen} solution of the Adams conjecture) and for $d=4k+1$ the group $\bP_{d+1}$ has order $2$ if $\Kerv_{d+1}=0$ and vanishes otherwise. Kervaire--Milnor showed that $\Kerv_d\neq0$ for $d=6,14$, Mahowald--Tangora \cite{MahowaldTangora} proved $\Kerv_{30}\neq0$, Barratt--Jones--Mahowald \cite{BarrattJonesMahowald} established $\Kerv_{62}\neq0$, Browder \cite{Browder} showed that $\Kerv_d=0$ if $d$ is not of the form $2^k-2$, and Hill--Hopkins--Ravenel \cite{HillHopkinsRavenel} proved that $\Kerv_{2^k-2}=0$ for $k\ge8$, leaving $2^7-2=126$ as the last unknown case.

\subsection{Parametrised Pontryagin--Thom theory}Let $M$ be a closed, connected, oriented manifold of dimension $2n$. Choose a Moore--Postnikov $n$-factorisation \[M\xlra{\ell_M}B\xlra{\theta}\BSO\] of its stable oriented tangent bundle; that is, a factorisation into an $n$-connected cofibration $\ell_M$ followed by an $n$-co-connected fibration $\theta$. This factorisation is unique up to weak equivalence under $M$ and over $\BSO$, and is called the \emph{tangential $n$-type of $M$}. Using this, the manifold $M$ naturally defines a class $[M,\ell_M]$ in the bordism group $\Omega^{\theta^\perp}_{2n}$ associated to its tangential $n$-type. We denote by $\theta_d\colon B_d\ra\BSO(d)$ the pullback of $\theta\colon B\ra\BSO$ along the canonical map $\BSO(d)\ra\BSO$.
For $d=2n$, this pullback induces a factorisation \begin{equation}\label{equation:unstablefactorisation}M\xlra{\overline{\ell}_M} B_{2n}\xlra{\theta_{2n}}\BSO(2n)\end{equation} of the unstable oriented tangent bundle of $M$, which can be seen to again be a Moore--Postnikov $n$-factorisation. We define $\MT\theta_d$ to be the Thom spectrum $\Thom(-{\theta_d}^*\gamma_d)$ of the inverse of the pullback of the canonical vector bundle $\gamma_d$ over $\BSO(d)$ along $\theta_d$. The topological monoid $\hAut(\theta_{d})$ of weak equivalences  $B_d\ra B_d$ over $\BSO(d)$ acts on ${\theta_d}^*\gamma_d$ by bundle automorphisms, inducing an action on the spectrum $\MT\theta_{d}$ and hence on its associated infinite loop space. The homotopy quotient $(\Omega^\infty\MT\theta_{d} )\dslash\hAut(\theta_{d})$ of this action is the target of the parametrised Pontryagin--Thom map which is subject of the following result of Galatius--Randal-Williams \cite[Cor.\,1.9]{GRWII}.

\begin{thm}[Galatius--Randal-Williams]\label{theorem:GRW}Let $M$ be a simply connected, closed, oriented manifold of dimension $2n\ge6$. There is a parametrised Pontryagin--Thom map
\[\BDiff(M)\lra(\Omega^\infty\MT\theta_{2n} )\dslash\hAut(\theta_{2n})\] which induces a homology isomorphism in degrees $*\le\frac{g(M)-3}{2}$ onto the path component hit.
\end{thm}
Here $g(M)$ denotes the genus of the manifold $M$, as defined in the introduction.

\begin{rem}Theorem~\ref{theorem:GRW} is the higher-dimensional analogue of a pioneering result for surfaces obtained by combining a classical homological stability result due to Harer \cite{Harer} with the celebrated theorem of Madsen--Weiss \cite{MadsenWeiss}.
\end{rem}
\begin{rem}Recent work of Friedrich \cite{Friedrich} can be used to strengthen Theorem~\ref{theorem:GRW} to manifolds that are not simply connected, but whose associated group ring $\bfZ[\pi_1M]$ has finite unitary stable rank (cf.\,the remark in the introduction).
\end{rem}
\subsection{The path components of $\MT\theta_{2n}$}\label{section:pathcomponents}
The Pontryagin--Thom correspondence identifies the group of path components $\pi_0\MT\theta_{2n}$ with the bordism group of closed smooth manifolds $N$ of dimension $2n$ together with a map $\ell\colon N\ra B_{2n}$ and a stable isomorphism $\varphi\colon TN\cong_s\ell^*\theta_{2n}^*\gamma_{2n}$, where a weak self-equivalence $[g]\in\pi_0\hAut(\theta_{2n})$ of $B_{2n}$ over $\BSO(2n)$ acts on such a triple $[N,\ell,\varphi]\in\pi_0\MT\theta_{2n}$ by replacing $\ell$ with $g\ell$ (see e.g.\,\cite[Thm A.2]{Ebert}). These bordism classes carry a canonical orientation induced from the universal orientation of $\gamma_{2n}$ by means of $\varphi$ and $\ell$, and this orientation is preserved by the $\pi_0\hAut(\theta_{2n})$-action. Note that the factorisation \eqref{equation:unstablefactorisation} provides a distinguished bordism class $[M,\overline{\ell}_M,\operatorname{id}]\in \pi_0\MT\theta_{2n}$  of the initial manifold $M$ defining $\theta$. There is a cofibre sequence \begin{equation}\label{equation:cofibresequence}\MT\theta_{d+1}\lra\Sigma_+^\infty B_{d+1}\lra\MT\theta_d\end{equation} for $d\ge0$ and a $(d+1)$-connected stabilisation map
\begin{equation}\label{equation:stabilisationmap}\Sigma^{d+1}\MT\theta_{d+1}\lra\M\theta^\perp,\end{equation} which in particular identifies $\pi_{-1}\MT\theta_{d+1}$ with the bordism group $\Omega_{d}^{\theta^\perp}$ (see \cite[Ch.\,3, 5]{GMTW}). For $d=2n$ and $d=2n+1$, this induces a morphism of the form
\begin{equation}\label{equation:pathcomponents}\pi_0\MT\theta_{2n}\lra\pi_0\Sigma_+^\infty B_{2n}\oplus\pi_{-1}\MT\theta_{2n+1}\cong\bfZ\oplus\Omega_{2n}^{\theta^\perp}\end{equation} whose source and target are naturally acted upon by $\pi_0\hAut(\theta_{2n})$; the action on the target fixes the $\bfZ$-summand and changes the $\theta$-structure on $\Omega_{2n}^{\theta^\perp}$.

\begin{lem}\label{lemma:injectivity}
The morphism \eqref{equation:pathcomponents} is injective and $\pi_0\MT\theta_{2n}$-equivariant.
\end{lem}
\begin{proof}The cofibre sequence \eqref{equation:cofibresequence} induces a diagram
\begin{center}
\begin{tikzcd}
\pi_0\Sigma_+^\infty B_{2n+1}\arrow[dr,dashed]\arrow[r]&\pi_0\MT\theta_{2n}\arrow[d]\arrow[r]&\Omega_{2n}^{\theta^\perp}\arrow[r]&0\\
&\pi_0\Sigma_+^\infty B_{2n}&&
\end{tikzcd}
\end{center}
with exact upper row, which can be described geometrically as follows (see e.g.\,the proof of \cite[Prop.\,1.2.7]{Gollinger}): the map $\bfZ\cong\pi_0\Sigma_+^\infty B_{2n+1}\ra \pi_0\MT\theta_{2n}$ sends a generator to $[S^{2n},\ell,\operatorname{id}]\in\pi_0\MT\theta_{2n}$ where $\ell$ is induced by the tangent bundle $S^{2n}\ra\BSO(2n)$ and its standard stable framing. The map $\pi_0\MT\theta_{2n}\ra \pi_0\Sigma_+^\infty B_{2n}\cong\bfZ$ maps a bordism class $[M,\ell,\varphi]$ to the evaluation $\langle e(\ell^*\theta_{2n}^*\gamma_{2n});[M]\rangle\in\bfZ$ of the Euler class. Finally, the map $\pi_0\MT\theta_{2n}\ra\Omega_{2n}^{\theta^\perp}$ sends a triple $[M,\ell,\varphi]$ to the class of $M$ equipped with the tangential $\theta$-structure induced by $\ell$ and $\varphi$. The claim regarding equivariance follows from noting that the action by $\pi_0\hAut(\theta_{2n})$ leaves $\langle e(\ell^*\theta_{2n}^*\gamma_{2n});[M]\rangle$ unchanged as it fixes the orientation. Since the dashed arrow is injective since it is up to isomorphism given by multiplication with $\langle e(TS^{2n}),[S^{2n}]\rangle=2$, the injectivity part of the claim is a consequence of exactness.\end{proof}

Justified by this, we identify $\pi_0\MT\theta_{2n}$ as a subgroup of $\bfZ\oplus\Omega_{2n}^{\theta^\perp}$ henceforth. The path components of $(\Omega^\infty\MT\theta_{2n}) \dslash\hAut(\theta_{2n})$ agree with the quotient \[\pi_0\MT\theta_{2n}/\pi_0\hAut(\theta_{2n})\subset \bfZ\times \Omega_{2n}^{\theta^\perp}/\pi_0\hAut(\theta_{2n})\] and the path component hit by the parametrised Pontryagin--Thom map  \[\BDiff(M)\lra(\Omega^\infty\MT\theta_{2n}) \dslash\hAut(\theta_{2n})\] is the one induced by the bordism class $[M,\overline{\ell}_M,\operatorname{id}]$, so it is represented by $(\chi(M),[M,\ell_{M}])\in\pi_0\MT\theta_{2n}$, where $\chi(M)\in\bfZ$ is the Euler characteristic of $M$ (cf.\,\cite[Ch.\,3]{GMTW}). We denote this path component of the homotopy quotient by $\big(\Omega^\infty\MT\theta_{2n}\dslash\hAut(\theta_{2n})\big)_M$ and the one of $\Omega^\infty\MT\theta_{2n}$ corresponding to $(\chi(M),[M,\ell_{M}])$ by $\Omega^\infty_{(M,\ell_M)}\MT\theta_{2n}$. Note that the inclusion $\Omega_{(M,\ell_M)}^\infty\MT\theta_{2n}\subset \Omega^\infty\MT\theta_{2n}$ induces a weak equivalence \begin{equation}\label{equation:equivalence}\Omega_{(M,\ell_M)}^\infty\MT\theta_{2n}\dslash\Stab(M,\ell_M)\simeq \big(\Omega^\infty\MT\theta_{2n}\dslash\hAut(\theta_{2n})\big)_M,\end{equation} where $\Stab(M,\ell_M)\subset\hAut(\theta_{2n})$ is the submonoid defined as the union of the path components of $\hAut(\theta_{2n})$ that fix $[M,\ell_{M}]\in\Omega_{2n}^{\theta^\perp}$.

\subsection{Proof of Theorem~\ref{theorem:maintheorem}}

By obstruction theory, the tangential $n$-type of a closed oriented $2n$-manifold $M$ depends only on the manifold $M\backslash \interior{D^{2n}}$ obtained by cutting out an embedded disc $D^{2n}\subset M$.\footnote{In fact, it only depends on the $n$-skeleton of the manifold.} This implies that the tangential $n$-type 
of the connected sum $M\sharp\Sigma$ of $M$ with a homotopy sphere $\Sigma\in\Theta_{2n}$ has the form \[M\sharp \Sigma \xlra{\ell_{M\sharp\Sigma}}B\xlra{\theta}\BSO\] for the same $B$ and $\theta$ as for $M$. Here $\ell_{M\sharp\Sigma}$ is the unique (up to homotopy) extension to $M\sharp \Sigma$ of the restriction of $\ell_M$ to $M\backslash \interior{D^{2n}}$, where $D^{2n}\subset M$ is the disc at which the connected sum was taken. The targets of the parametrised Pontryagin--Thom maps of Theorem~\ref{theorem:GRW} for $M$ and $M\sharp\Sigma$ thus agree,
\[\BDiff(M)\lra(\Omega^\infty\MT\theta_{2n}) \dslash\hAut(\theta_{2n})\lla \BDiff(M\sharp\Sigma),\] and both maps induce an isomorphism on homology in a range of degrees onto the respective path components hit. However, different path components of the target space will usually have nonisomorphic homology; nonetheless, we have the following lemma comparing the two components in question.

\begin{lem}\label{lemma:zigzag}For a closed, oriented $2n$-manifold $M$ and $\Sigma\in\Theta_{2n}$, there is a zig-zag between\[\big(\Omega^\infty\MT\theta_{2n} \dslash\hAut(\theta_{2n})\big)_M\quad\text{and}\quad\big(\Omega^\infty\MT\theta_{2n} \dslash\hAut(\theta_{2n})\big)_{M\sharp\Sigma},\] inducing homology isomorphisms with $\bfZ[\frac{1}{k}]$-coefficients, $k$ being the order of $\Sigma$.
\end{lem}
\begin{proof}The stable oriented tangent bundle $\Sigma\ra\BSO$ of $\Sigma$ lifts to a unique tangential $\theta$-structure $\ell_\Sigma$ of $\Sigma$ and, by obstruction theory, the standard bordism between the connected sum $M\sharp \Sigma$ and the disjoint union of $M$ and $\Sigma$ extends to a bordism respecting the $\theta$-structure. This gives $[M\sharp\Sigma,\ell_{M\sharp \Sigma}]=[M,\ell_M]+[\Sigma,\ell_\Sigma]$ in $\Omega^{\theta^\perp}_{2n}$ and by the same argument, we obtain a morphism $\Theta_{2n}\ra\Omega_{2n}^{\theta^\perp}$ sending $\Sigma\in\Theta_{2n}$ to $[\Sigma,\ell_\Sigma]$, from which we conclude that the order of $[\Sigma,\ell_\Sigma]\in\Omega_{2n}^{\theta^\perp}$ divides $k$. As the Euler characteristics of $M$ and $M\sharp\Sigma$ agree, we have \[k\cdot(\chi(M\sharp\Sigma),[M\sharp\Sigma,\ell_M\sharp\Sigma])=k\cdot(\chi(M),[M,\ell_M])\] in $\pi_0\MT\theta_{2n}$. Multiplication by $k$ in the infinite loop space $\Omega^\infty\MT\theta_{2n}$ hence maps the path components $(\chi(M),[M,\ell_M])$ and $(\chi(M\sharp\Sigma),[M\sharp\Sigma,\ell_{M\sharp\Sigma}])$ to the same path component, denoted $\Omega^\infty_{k\cdot (M,\ell_M)}\MT\theta_{2n}$. As observed above, the homotopy sphere $\Sigma$ has a unique $\theta$-structure compatible with its orientation, so the class $[\Sigma,\ell_\Sigma]$ in $\Omega^{\theta^\perp}_{2n}$ is fixed by the action of $\hAut(\theta)$, which in turn implies $\Stab(M,\ell_M)=\Stab(M\sharp\Sigma,\ell_{M\sharp\Sigma})\subset\hAut(\theta_{2n})$. Since multiplication by $k$ in $\Omega^\infty\MT\theta_{2n}$ is $\hAut(\theta_{2n})$-equivariant, we obtain an induced zig-zag \begin{center}
\begin{tikzcd}[column sep=-2.5cm]
&(\Omega^\infty_{k\cdot (M,\ell_M)}\MT\theta_{2n})\dslash\Stab(M,\ell_M)&\\
(\Omega_{(M,\ell_M)}^\infty\MT\theta_{2n})\dslash\Stab(M,\ell_M)\arrow[ur,"k\cdot-"]&&(\Omega_{(M\sharp\Sigma,\ell_{M\sharp\Sigma})}^\infty\MT\theta_{2n})\dslash\Stab(M\sharp\Sigma,\ell_{M\sharp\Sigma})\arrow[ul,"k\cdot-",swap].
\end{tikzcd}
\end{center}
The corresponding zig-zag between the respective path components of $\Omega^\infty\MT\theta_{2n}$ before taking homotopy quotients is given by multiplication by $k$, so induces an isomorphism on homology with $\bfZ[\frac{1}{k}]$-coefficients. The claim now follows from a comparison of the Serre spectral sequences of the homotopy quotients, together with the equivalences \eqref{equation:equivalence}. \end{proof}

By taking connected sums with $\Sigma\in\Theta_{2n}$ and its inverse, one sees that the genera of $M$ and $M\sharp \Sigma$ coincide, so Theorem~\ref{theorem:maintheorem} is a direct consequence of Theorem~\ref{theorem:GRW} and Lemma~\ref{lemma:zigzag}.

\section{The collar twist}
In this section, we examine the homotopy fibre sequence
\begin{equation}
\label{equation:fibresequencediffeogroups}
\BDiffuo(M,D^{d})\lra\BDiff(M,*)\lra \BSO(d)
\end{equation}for a closed oriented $d$-manifold $M$ with an embedded disc $D^d\subset M$, induced by the fibration $\der\colon \Diff(M,*)\ra\SO(d)$ of topological groups that assigns a diffeomorphism fixing a basepoint $*\in M$ its derivative at this point. Its fibre at the identity is the subgroup $\Diffuo(M,T_*M)\subset\Diff(M,*)$ of diffeomorphisms fixing the tangent space at $*\in M$, which contains all diffeomorphisms that pointwise fix an embedded disc $D^d\subset M$ with centre $*\in M$ as a homotopy equivalent subgroup $\Diffuo(M,D^{d})\subset\Diffuo(M,T_*M)$, explaining the sequence \eqref{equation:fibresequencediffeogroups}. We are particularly interested in the effect on homotopy groups of the map $t\colon \SO(d)\ra\BDiffuo(M,D^d)$ induced from looping \eqref{equation:fibresequencediffeogroups} once. Fixing a collar $c\colon [0,1]\times S^{d-1}\ra M\backslash \text{int}(D^{d})$ satisfying $c^{-1}(\partial D^{d})=\{1\}\times S^{d-1}$, the map $t$ agrees with the delooping of the morphism $\Omega\SO(d)\ra\Diffuo(M,D^{d})$ sending a smooth loop $\gamma\in\Omega\SO(d)$ to the diffeomorphism that is the identity on $D^{d}$ as well as outside the collar and
\[\mapnoname{[0,1]\times S^{d-1}}{[0,1]\times S^{d-1}}{(t,x)}{(t,\gamma(t)x)}\] on the collar.\footnote{This is because $\gamma\in\Omega\SO(d)$ lifts along $\Diff(M,*)\ra\SO(d)$ to $\tilde{\gamma}\colon [0,1]\ra\Diff(M,*)$ satisfying $\tilde{\gamma}(0)=\operatorname{id}$ and $\tilde{\gamma}(1)\in\Diffuo(M,D^{d})$ by defining $\tilde{\gamma}(t)$ on $D^d$ by acting with $\gamma(t)$, on $[0,1]\times S^{d-1}$ by sending $(s,x)\in[0,1]\times S^{d-1}$ to $(s,\gamma(s) x)$ if $s\le t$ and to $(s,\gamma(t)x)$ if $s>t$, and to be the identity elsewhere. Strictly speaking, $\tilde{\gamma}$ is only a path of homeomorphisms, but this can easily be remedied by using a suitable family of bump functions.}
Inspired by this geometric description, we call the map \[t\colon\SO(d)\lra\BDiffuo(M,D^{d})\] the \emph{collar twist of $M$}.

\begin{rem}In dimension $d=2$, the collar twist $\SO(2)\rightarrow\BDiffuo(M,D^{2})$ is clearly trivial on homotopy groups except on fundamental groups, where the induced map $\bfZ\ra\pi_0\BDiffuo(M,D^{2})$ is given by a Dehn twist on $[0,1]\times S^{1}\subset M$.
\end{rem}

\subsection{Triviality of the collar twist}As indicated by the following lemma, the collar twist serves as a measure for the degree of linear symmetry of the underlying manifold.

\begin{lem}\label{lemma:action}Let $M$ be a closed oriented $d$-manifold that admits a smooth orientation-preserving action of $\SO(k)$ with $k\le d$. If the action has a fixed point $*\in M$ whose tangential representation is the restriction of the standard representation of $\SO(d)$ to $\SO(k)$, then the long exact sequence induced by the fibre sequence \eqref{equation:fibresequencediffeogroups} reduces to split short exact sequences \[0\lra\pi_i\BDiffuo(M,D^{d})\lra\pi_i\BDiff(M,*)\lra\pi_{i-1}\SO(d)\lra0\] for $i\le k-1$. 
In particular, the collar twist is trivial on homotopy groups in this range.
\end{lem}
\begin{proof}
On the subgroup $\SO(k)\subset\SO(d)$, the $\SO(k)$-action on $M$ provides a left-inverse to the derivative map $\der\colon\Diff(M,*)\ra \SO(d)$. As the inclusion $\SO(k)\subset\SO(d)$ is $(k-1)$-connected, we conclude that the derivative map is surjective on homotopy groups in degree $k-1$ and split surjective in lower degrees, which implies the result.
\end{proof}
The action of $SO(d)$ on the standard sphere $S^{d}$ by rotation along an axis satisfies the assumption of the lemma, so the collar twist of $S^d$ is trivial on homotopy groups up to degree $d-1$. In fact, in this case it is even nullhomotopic. Another family of manifolds that admit a smooth action of $\SO(k)$ as in the lemma is given by the $g$-fold connected sums \[W_g=\sharp^g(S^n\times S^n).\] Indeed, by \cite[Prop.\,4.3]{GGRW}, there is a smooth $\SO(n)\times\SO(n)$-action on $W_g$ whose restriction to a factor can be seen to provide an action of $\SO(n)$ as wished. This action has an alternative description, kindly pointed out to us by Jens Reinhold: consider the action of $\SO(n)$ on $S^n\times S^n$ by rotating the first factor around the vertical axis. Both the product of the two upper hemispheres of the two factors and the product of the two lower ones are preserved by the action and are, after smoothing corners, diffeomorphic to a disc $D^{2n}$, acted upon via the inclusion $\SO(n)\subset\SO(2n)$ followed by the standard action of $\SO(2n)$ on $D^{2n}$. Taking the $g$-fold equivariant connected sum of $S^n\times S^n$ using these discs results in an action of $\SO(n)$ on $W_g$ as in Lemma~\ref{lemma:action} and thus has the following as a consequence.

\begin{cor}\label{corollary:collartwistingtrivial}$t_*\colon\pi_i\SO(2n)\ra\pi_i\BDiffuo(W_g,D^{2n})$ is trivial for $i\le n-1$.
\end{cor}

\subsection{Detecting the collar twist in bordism}
\label{section:detectingcollartwisting} Recall that a manifold $M$ is called (stably) \emph{$n$-parallelisable} if it admits a tangential $\langle n\rangle$-structure for the $n$-connected cover $\BO\langle n\rangle\ra \BO$. This map factors for $n\ge1$ over $\BSO$ and obstruction theory shows that there is a unique (up to homotopy) equivalence $\BO\langle n\rangle^\perp\simeq \BO\langle n\rangle$ over $\BSO$, so tangential and normal $\langle n\rangle$-structures of a manifold $M$ are naturally equivalent. For oriented manifolds $M$, we require that $\langle n\rangle$-structures on $M$ are compatible with the orientation; that is, they lift the oriented stable tangent bundle $M\ra\BSO$ of $M$. Another application of obstruction theory shows that the map $\BSO\times\BSO\ra\BSO$ classifying the external sum of oriented stable vector bundles is up to homotopy uniquely covered by a map $\BO\langle n\rangle\times\BO\langle n\rangle\ra\BO\langle n\rangle$ turning $\MO\langle n\rangle$ into a homotopy commutative ring spectrum.

In the following, we describe a method to detect the nontriviality of certain collar twists. For this, we restrict our attention to oriented manifolds $M$ that are $(n-1)$-connected, $n$-parallelisable, and $2n$-dimensional. These manifolds have by obstruction theory a unique $\langle n \rangle$-structure $\ell_M\colon M\ra \BO\langle n \rangle$, so they determine a canonical class $[M,\ell_M]$ in the bordism group $\Omega_{2n}^{\langle n\rangle}$. The examples we have in mind are connected sums $W_g\sharp \Sigma$ of $W_g$ with $\Sigma\in\Theta_{2n}$.

\begin{rem}In fact, all $(n-1)$-connected $n$-parallelisable $2n$-manifolds $M$ with positive genus and vanishing signature are of the form $W_g\sharp \Sigma$ for a homotopy sphere $\Sigma$, except possibly those in dimension $4$ and the Kervaire invariant one dimensions (see Section~\ref{section:homotopyspheres}). Indeed, the vanishing of the signature implies that $M$ is stably parallelisable and therefore by the work of Kervaire--Milnor \cite{KervaireMilnor} framed bordant to a homotopy sphere. From this, an application of Kreck's modified surgery \cite[Thm\,C--D]{KreckSurgery} shows the claim.
\end{rem}

For an oriented $(n-1)$-connected $n$-parallelisable $2n$-manifold $M$, define morphisms
\begin{equation}
\label{equation:detectionmorphism}\Phi_i\colon\pi_{i}\BDiffuo(M,D^{2n})\lra\Omega^{\langle n\rangle}_{2n+i}\end{equation} as follows: a class $[\varphi]\in\pi_{i}\BDiffuo(M,D^{2n})$ classifies a smooth fibre bundle \[M\ra E_\varphi\ra S^i,\] together with the choice of a trivialised $D^{2n}$-subbundle $S^i\times D^{2n}\subset E_\varphi$. The latter induces a trivialisation of the normal bundle of $S^i=S^i\times\{0\}$ in $E_\varphi$. Using this, a given $\langle n \rangle$-structure on $E_\varphi$ induces an $\langle n\rangle$-structure on the embedded $S^i$ and conversely, obstruction theory shows that every $\langle n \rangle$-structure on $S^i$ is induced by a unique $\langle n \rangle$-structure on $E_\varphi$ in this manner. We can hence define $\Phi_i$ by sending a homotopy class $[\varphi]$ to the total space $E_\varphi$ of the associated bundle, together with the unique $\langle n\rangle$-structure extending the canonical $\langle n\rangle$-structure on $S^i$ induced by the standard stable framing of $S^i$.

\begin{rem}\label{remark:mapofspaces}The morphisms $\Phi_i$ are induced by a map of spaces\[\BDiffuo(M,D^{2n})\lra\Omega^{\infty}\Sigma^{-2n}\MO\langle n\rangle,\] which results from composing the parametrised Pontryagin--Thom map $\BDiffuo(M,D^{2n})\ra\Omega^{\infty}\MTO(2n)\langle n\rangle$  (see e.g.\,\cite[Thm\,1.2]{GRWstable}) with the $2n$-fold desuspension $\MTO(2n)\langle n\rangle\ra\Sigma^{-2n}\MO\langle n\rangle$ of the stabilisation map \eqref{equation:stabilisationmap} for $\MTO(2n)\langle n\rangle$.\end{rem}

\begin{lem}\label{lemma:diagram}The following diagram is commutative
\begin{center}
\begin{tikzcd}
\pi_i\SO(2n)\arrow[r,"t_*"]\arrow[d, "J",swap]&\pi_i\BDiffuo(M,D^{2n})\arrow[d,"\Phi_i"]\\
\Omega_i^{\langle n\rangle}\arrow[r,"{-\cdot[M,\ell_M]}"]&\Omega_{2n+i}^{\langle n\rangle},
\end{tikzcd}
\end{center} the left vertical map being the $J$-homomorphism followed by the map from framed to $\langle n\rangle$-bordism, and the bottom map the multiplication by $[M,\ell_M]\in\Omega_{2n}^{\langle n\rangle}$.
\end{lem}
\begin{proof}
Recall from the beginning of the chapter that the map $t$ is given by twisting a collar $[0,1]\times S^{2n-1}\subset M\backslash\interior{D^{2n}}$. The composition of $t_*$ with $\Phi_i$ maps a class in $\pi_i\SO(2n)$, represented by a smooth map $\varphi\colon S^i\ra\SO(2n)$, to the bordism class of a certain manifold $E_{t(\varphi)}$, equipped with a particular $\langle n\rangle$-structure. The manifold $E_{t(\varphi)}$ is constructed using a clutching function that twists the collar using $\varphi$ and is constant outside of it. Its associated $\langle n\rangle$-structure is the unique one that extends the canonical one on $S^i$ via the given trivialisation of its normal bundle. Untwisting the collar using the standard $\SO(2n)$-action on the disc $D^{2n}$ yields a diffeomorphism $E_{t(\varphi)}\cong S^i\times M$, which coincides with the twist \[\mapnoname{D^{2n}\times S^i}{D^{2n}\times S^i}{(x,t)}{(\varphi(t)x,t)}\] on the canonically embedded $S^i\times D^{2n}$ and is the identity on the other component of the complement of the twisted collar $S^i\widetilde{\times}([0,1]\times S^{2n-1})\subset E_\varphi$. Hence, the $\langle n\rangle$-structure on $S^i\times M$ induced from the one on $E_{t(\varphi)}$ via this diffeomorphism coincides with the product of the twist of the canonical $\langle n\rangle$-structure on $S^i$ by $\varphi$ with the unique one $\ell_M$ on $M$. By the bordism description of the $J$-homomorphism, this implies the claim.
\end{proof}
Lemma~\ref{lemma:diagram} serves us to detect the nontriviality of the collar twist \[t\colon\SO(2n)\lra\BDiffuo(M,D^{2n}).\] Indeed, if there is a nontrivial element in the subgroup $\operatorname{im}(J)_i\cdot [M,\ell_M]$ of $\pi_{2n+i}\MO{\langle n\rangle}$, then Lemma~\ref{lemma:diagram} implies that the collar twist of $M$ is nontrivial on homotopy groups in degree $i$. The $g$-fold connected sum $W_g$ is the boundary of the parallelisable handlebody $\natural^{g}(D^{n+1}\times S^n)$ and is thus trivial in framed bordism and so also in $\langle n\rangle$-bordism, giving \[[W_g\sharp\Sigma,\ell_{W_g\sharp\Sigma}]=[W_g,\ell_{W_g}]+[\Sigma,\ell_{\Sigma}]=[\Sigma,\ell_\Sigma]\] in $\Omega_{2n}^{\langle n\rangle}$ for all $\Sigma\in\Theta_{2n}$. This proves the first part of the following proposition.

\begin{prop}\label{proposition:nontrivialitycondition}Let $\Sigma\in\Theta_{2n}$ be a homotopy sphere and $g\ge0$. If the subgroup \[\operatorname{im}(J)_i\cdot [\Sigma,\ell_\Sigma]\subset\pi_{2n+i}\MO{\langle n\rangle}\] is nontrivial for some $i\ge 1$, then the abelianisation of the morphism \[t_*\colon\pi_i\SO(2n)\lra\pi_i\BDiffuo(W_g\sharp\Sigma,D^{2n})\] is nontrivial. Assuming $2n\ge 6$, the converse holds for $i=1$ and $g\ge2$.
\end{prop}
\begin{proof}We are left to prove the second claim. The group $\pi_1\BDiffuo(W_g\sharp\Sigma,D^{2n})$ is the group of isotopy classes of orientation-preserving diffeomorphisms of $W_g\sharp\Sigma$ that fix the embedded disc $D^{2n}$ pointwise. Letting such diffeomorphisms act on the middle-dimensional homology group $\oH_n(W_g\sharp\Sigma;\bfZ)$ provides a morphism $\pi_1\BDiffuo(W_g\sharp\Sigma,D^{2n})\ra\aut(Q_{W_g\sharp\Sigma})$ to the subgroup $\aut(Q_{W_g\sharp\Sigma})\subset\operatorname{GL}(\oH_n(W_g\sharp\Sigma;\bfZ))$ of automorphisms that preserve Wall's \cite{Wall2n} quadratic form associated to the $(n-1)$-connected $2n$-manifold $W_g\sharp\Sigma$. Together with the morphism \eqref{equation:detectionmorphism} in degree $1$, this combines to a map
\begin{equation} \label{euqation:abelianisation}\pi_1\BDiffuo(W_g\sharp\Sigma,D^{2n})\lra\aut(Q_{W_g\sharp\Sigma})\times \Omega_{2n+1}^{\langle n\rangle},\end{equation} which induces an isomorphism on abelianisations by \cite[Cor.\,2.4, Cor.\,4.3 (i), Thm 4.4]{Krannich}, assuming $g\ge2$ and $2n\ge 6$.\footnote{The results of \cite{Krannich} are stated for $\pi_1\BDiffuo(W_g,D^{2n})$, but the given proof goes through for $\pi_1\BDiffuo(W_g\sharp\Sigma,D^{2n})$ for any homotopy sphere $\Sigma\in\Theta_{2n}$.} The isotopy classes in the image of $\pi_1\SO(2n)\ra\pi_1\BDiffuo(W_g\sharp\Sigma,D^{2n})$ are supported in a disc and hence act trivially on homology. Consequently, the collar twist is for $g\ge 2$ nontrivial on abelianised fundamental groups if and only if it is nontrivial after composition with $\pi_1\BDiffuo(W_g\sharp\Sigma,D^{2n})\ra\Omega^{\langle n\rangle}_{2n+1}$, so the claim is a consequence of Lemma~\ref{lemma:diagram} and the discussion prior to this proposition.\end{proof}

\begin{rem}\label{remark:remainingcases}For $g=1$, the second part of Proposition~\ref{proposition:nontrivialitycondition} remains valid for $n$ even and for $n=3,7$, since \eqref{euqation:abelianisation} still induces the abelianisation in these cases by \cite[Thm 4.4, Lem.\, 3.9]{Krannich}. The cited results moreover show that \eqref{euqation:abelianisation} no longer induces the abelianisation for $g=1$ and $n\neq1,3,7$ odd, but they can be used in combination with \cite[Lem.\,4, Thm 3 c)]{Kreck} to show that the conclusion of the previous proposition is valid in these cases nevertheless, at least when assuming a conjecture of Galatius--Randal-Williams \cite[Conj.\,A]{GRWabelian}. For $g=0$, the second part of the proposition fails in many cases, which can be seen from Proposition~\ref{proposition:trivialcollartwisting} and its proof (cf.\,Remark~\ref{remark:fails}).
\end{rem}

\begin{rem} \label{remark:trivialinhighdegree}
\ 
\begin{enumerate}\item Since the image of the stable $J$-homomorphism is cyclic, the condition in the previous is equivalent to the nonvanishing of a single class, e.g.\,$\eta\cdot [\Sigma,\ell_\Sigma]\in\pi_{2n+1}\MO\langle n\rangle$ for $i=1$ or $\nu\cdot [\Sigma,\ell_\Sigma]\in\pi_{2n+3}\MO\langle n\rangle$ for $i=3$.
\item By obstruction theory, the sphere $S^i$ has a unique $\langle n\rangle$-structure for $i\ge n+1$, so the left vertical morphism in the diagram of Lemma~\ref{lemma:diagram} is trivial in this range. Hence, the morphism $\pi_i\BDiffuo(M,D^{2n})\ra\Omega^{\langle n \rangle}_{2n+i}$ can only detect the possible nontriviality of the collar twist in low degrees relative to the dimension. Another consequence of this uniqueness is that the image $\im(J)_i$ of the $J$-homomorphism is contained in the kernel of the canonical morphism from framed to $\langle n\rangle$-bordism for $i\ge n+1$.
\end{enumerate}
\end{rem}
\begin{rem}\label{remark:obstructions}
\ 
\begin{enumerate}
\item
Combining Lemmas~\ref{lemma:action} and \ref{lemma:diagram}, we see that the subgroups $\operatorname{im}(J)_i\cdot[M]\subset\pi_{2n+i}\MO\langle n\rangle$ obstruct smooth $\SO(k)$-actions on $M$ with a certain fixed point behaviour.
\item
In the case $g=0$, the first part of Proposition~\ref{proposition:nontrivialitycondition} is closely related to a result obtained by Schultz \cite[Thm\,1.2]{Schultz}. In combination with methods developed by Hsiang--Hsiang (see e.g.\,\cite{HsiangHsiang}), Schultz used this to bound the degree of symmetry of certain homotopy spheres $\Sigma$, i.e.\,the maximum of the dimensions of compact Lie groups acting effectively on $\Sigma$.
\end{enumerate}
\end{rem}
In what follows, we attempt to detect nontrivial elements in the subgroup \begin{equation}\label{equation:product}\operatorname{im}(J)_i\cdot [\Sigma,\ell_\Sigma]\subset\pi_{2n+i}\MO{\langle n\rangle}\end{equation} for homotopy spheres $\Sigma\in\Theta_{2n}$ by mapping $\MO\langle n\rangle$ to ring spectra whose rings of homotopy groups are better understood. A natural choice of such spectra are $\MO\langle m\rangle$ for small $m\le n$ via the canonical map $\MO\langle n\rangle\ra\MO\langle m\rangle$. As homotopy spheres are stably parallelisable, the elements $[\Sigma,\ell_\Sigma]$ are in the image of the unit $\bfS\ra\MO\langle n\rangle$. Conversely, outside of the Kervaire invariant one dimensions $2,6,14,30,62$, and possibly $126$, the work of Kervaire--Milnor \cite{KervaireMilnor} implies that all elements in $\pi_*\bfS$ are represented by homotopy spheres (see Section~\ref{section:homotopyspheres}). Since the image of the unit $\bfS\ra\MO\langle 2\rangle=\MSpin$ lies in degrees $8k+1$ and $8k+2$ for $k\ge0$ by \cite[Cor.\,2.7]{AndersonBrownPeterson}, we cannot detect nontrivial elements in \eqref{equation:product} by relying on spin bordism. As a result, since $\MO\langle2\rangle=\MO\langle3\rangle$, the smallest value of $m$ such that $\MO\langle m\rangle$ might possibly detect nontrivial such elements is $4$. Via the string orientation $\MString\ra\tmf$ (see e.g.\,\cite[Ch.\,10]{tmf}), the Thom spectrum $\MO\langle 4\rangle=\MString$ maps to the spectrum $\tmf$ of topological modular forms, whose ring of homotopy groups has been computed completely (see e.g.\,\cite[Ch.\,13]{tmf}). It is concentrated at the primes $2$ and $3$ and has a certain periodicity of degree $192$ at the prime $2$ and of degree $72$ at the prime $3$. The Hurewicz image of $\tmf$, i.e.\,the image of the composition \[\bfS\lra\MSO\langle n\rangle\lra\MString\lra\tmf\] on homotopy groups, is known at the prime 3 \cite[Ch.\,13]{tmf} and determined in work in progress by Behrens--Mahowald--Quigley at the prime $2$. It is known to contain many periodic families of nontrivial products with $\eta\in\pi_1\bfS$ and $\nu\in\pi_3\bfS$ which provide infinite families of homotopy spheres $\Sigma$ for which \eqref{equation:product} is nontrivial. The following proposition carries out this strategy in the lowest dimensions possible.

\begin{prop}\label{proposition:nontrivialproducts}In the following cases, there exists $\Sigma\in\Theta_{2n}$ for which the subgroup $\operatorname{im}(J)_i\cdot [\Sigma,\ell_\Sigma]$ of $\pi_{2n+i}\MO{\langle n\rangle}$ is nontrivial at the prime $p$:
\begin{enumerate}
\item for $p=2$ in degree $i=1$ and all dimensions $2n\equiv 8\Mod{192}$,
\item for $p=2$ in degree $i=3$ and all dimensions $2n\equiv 14\Mod{192}$, and
\item for $p=3$ in degree $i=3$ and all dimensions $2n\equiv 10\Mod{72}$.
\end{enumerate}
Consequently, in the respective cases, the abelianisation of the morphism
\[t_*\colon\pi_i\SO(2n)\lra\pi_i\BDiffuo(W_g\sharp\Sigma,D^{2n})\] is nontrivial at the prime $p$ for all $g\ge0$.
\end{prop}
\begin{proof}During the proof, we make use of various well-known elements in $\pi_*\bfS$, referring to \cite{Ravenel} for their definition (c.f.\,Thm 1.1.14 loc.\,cit.). From the discussion above, we conclude in particular that all elements in $\pi_{2n}\bfS$ for $2n$ as in one of the three cases of the claim can be represented by homotopy spheres. The elements $\varepsilon=\overline{\nu}-\eta\sigma\in\pi_8\bfS$ and $\kappa\in\pi_{14}\bfS$, respectively, generate $\Coker(J)$ in these dimensions and give rise to $192$-periodic families in $\pi_*\bfS$ at the prime $2$ whose elements are nontrivial when multiplied with $\eta\in\operatorname{im}(J)_1$ and $\nu\in\operatorname{im}(J)_3$, respectively (see e.g.\,\cite[Ch.\,11]{HopkinsMahowald}). All of these products are detected in $\tmf$ and hence are also nontrivial in $\MO\langle n\rangle$. This proves the first and the second case. To prove the third one, we make use of a $72$-periodic family in $\pi_*\bfS$ at the prime $3$ generated by the generator $\beta_1\in\pi_{10}\bfS$ of the $3$-torsion of $\Coker(J)_{10}$. The products of the elements in this family with $\nu\in\pi_3\bfS$ are nontrivial and are detected in $\tmf$ (see e.g.\,\cite[Ch.\,13]{tmf}). This proves the first part of the proposition. The second part is implied by the first by means of Proposition~\ref{proposition:nontrivialitycondition}.
\end{proof}
From work of Kreck \cite{Kreck}, we derive the existence of homotopy spheres $\Sigma$ for which the map $t_*\colon\pi_1\SO(2n)\ra\pi_1\BDiffuo(W_g\sharp\Sigma,D^{2n})$ induced by the collar twist is nontrivial, but becomes trivial upon abelianisation.

\begin{prop}\label{proposition:trivialcollartwisting}For all $k\ge1$, there are $\Sigma\in\Theta_{8k+2}$ for which \[t_*\colon\pi_1\SO(8k+2)\lra\pi_1\BDiffuo(W_g\sharp\Sigma,D^{8k+2})\] is nontrivial for $g\ge0$, but vanishes in the abelianisation for $g\ge2$.
\end{prop}
\begin{proof}The Kervaire--Milnor morphism $\Theta_{d}\ra\Coker(J)_d$ is an isomorphism in dimensions of the form $d=8k+2$ (see Section~\ref{section:homotopyspheres}), so there is a unique homotopy sphere $\Sigma_\mu\in\Theta_{8k+2}$ corresponding to the class of Adams' nontrivial element $\mu_{8k+2}\in\pi_{8k+2}\bfS$ (see \cite[Thm\,1.8]{Adams}). By \cite[Lem.\,4]{Kreck}, the morphism $t_*$ is trivial if and only if a certain homotopy sphere $\Sigma_{W_g\sharp\Sigma_\mu}\in\Theta_{8k+3}$ vanishes. Combining \cite[Lem.\,3 a), b)]{Kreck} with the fact that $W_g$ bounds the parallelisable handlebody $\natural^g(D^{4k+2}\times S^{4k+1})$, it follows that $\Sigma_{W_g\sharp\Sigma_\mu}$ agrees with $\Sigma_{W_0\sharp\Sigma_\mu}=\Sigma_{\Sigma_\mu}$. As $\im(J)_{8k+2}=0$ by Bott periodicity, the homotopy sphere $\Sigma_\mu$ has a unique stable framing, so an application of \cite[Thm 3 c)]{Kreck} shows that $\Sigma_{\Sigma_\mu}$ is trivial if and only if $\eta\cdot \mu_{8k+2}\in\pi_{8k+3}\bfS$ vanishes. But this product is a nontrivial element of order $2$ in the image of $J$ (see e.g.\,\cite[Thm\,5.3.7]{Ravenel}), so the first part follows. To prove the second, note that since $\eta\cdot \mu_{8k+2}\in\im(J)_{8k+3}$, the class $\eta\cdot[\Sigma_\mu,\ell_{\Sigma_\mu}]\in\pi_{8k+3}\MO\langle 4k+1\rangle$ is trivial (see Remark~\ref{remark:trivialinhighdegree} ii)), which implies the claim by the second part of Proposition~\ref{proposition:nontrivialitycondition}.
\end{proof}

\begin{rem}\label{remark:fails}\ 
\begin{enumerate}
\item For $g=1$, the conclusion of the second part of Proposition~\ref{proposition:trivialcollartwisting} still holds in most cases (see Remark~\ref{remark:remainingcases}), but has to fail for $g=0$ since $\pi_1\BDiffuo(\Sigma,D^{2n})$ is abelian. 
\item Recall that $\pi_1\BDiffuo(W_g\sharp\Sigma,D^{2n})$ is the group of isotopy classes of diffeomorphisms of $W_g\sharp\Sigma$ that fix an embedded disc $D^{2n}$ pointwise. The image of $t_*\colon\pi_1\SO(2n)\ra\pi_1\BDiffuo(W_g\sharp\Sigma,D^{2n})$ is generated by a single diffeomorphism $t(\eta)$, which is the higher-dimensional analogue of a Dehn twist for surfaces. It twists a collar around $D^{2n}$ by a generator in $\pi_1\SO(2n)$ and is the identity elsewhere. Consequently, the morphism $t_*$ is trivial if and only if $t(\eta)$ is isotopic to the identity, and it is trivial in the abelianisation if and only $t(\eta)$ is isotopic to a product of commutators.
\end{enumerate}
\end{rem}

\section{Detecting exotic smooth structures in diffeomorphism groups}
Evaluating diffeomorphisms of a closed oriented $d$-manifold $M$ at a basepoint $*\in M$ induces a homotopy fibre sequence
\begin{equation}\label{equation:evalutionsequence}M\lra\BDiff(M,*)\lra\BDiff(M),\end{equation} from which we see that for highly connected $M$ (such as $W_g\sharp\Sigma$), the map $\BDiff(M,*)\ra\BDiff(M)$ is also highly connected. Studying the space $\BDiff(M,*)$ instead of $\BDiff(M)$ is advantageous as it fits into the homotopy fibre sequence \eqref{equation:fibresequencediffeogroups} of the form \[\BDiffuo(M,D^{d})\lra\BDiff(M,*)\lra\BSO(d)\] whose base space is simply connected. Evidently, the rightmost space in this sequence is not affected by replacing $M$ by $M\sharp\Sigma$ for homotopy spheres $\Sigma$, and the following lemma shows that the same holds for the leftmost space as well.

\begin{lem}\label{lemma:relativediffisomorphic}For an oriented manifold $M$ of dimension $d\ge5$ and a homotopy sphere $\Sigma$ in $\Theta_{d}$, there is an isomorphism of topological groups \[\Diffuo(M,D^{d})\cong\Diffuo(M\sharp \Sigma,D^{d}).\]
\end{lem}

\begin{proof}
The group $\Diffuo(M\sharp \Sigma,D^{d})$ can be described equivalently as the group of diffeomorphisms of $M\sharp \Sigma\backslash \interior{D^{d}}$ that extend over the disc $D^{d}\subset M\sharp\Sigma$ by the identity. Manifolds of dimension $d\ge 5$ with nonempty boundary are insensitive to taking the connected sum with a homotopy sphere, so there is a diffeomorphism  $M\sharp \Sigma\backslash{\interior{D^{d}}}\cong M\backslash \interior{D^{d}}$. This diffeomorphism does not necessarily preserve the boundary, but conjugation with it nevertheless induces an isomorphism as claimed. \end{proof}

\begin{rem}By Lemma~\ref{lemma:relativediffisomorphic}, the total space of the homotopy fibre sequence \[\Fr(M)\lra\BDiffuo(M,D^{d})\lra\BDiff(M)\] involving the frame bundle $\Fr(M)$ does not change when replacing $M$ by $M\sharp\Sigma$ for homotopy spheres $\Sigma$. Somewhat surprisingly, the same holds for $\Fr(M)$ (see e.g.\,the proof of \cite[Thm\,1.2]{CrowleyZvengrowski}), but we shall not make use of this fact.
\end{rem}
\begin{proof}[Proof of Theorem~\ref{theorem:H1} and Corollary~\ref{corollary:H1}]Using the fibrations  \eqref{equation:fibresequencediffeogroups} and \eqref{equation:evalutionsequence}, we obtain an exact sequence of the form
\[\bfZ/2\cong\pi_1\SO(2n)\xlra{t_*}\pi_1\BDiffuo(W_g\sharp\Sigma,D^{2n})\xlra{\iota_*}\pi_1\BDiff(W_g\sharp \Sigma)\lra0.\] The first map is trivial in the case of the standard sphere $\Sigma=S^{2n}$ by Corollary~\ref{corollary:collartwistingtrivial}, a fact which, combined with Lemma~\ref{lemma:relativediffisomorphic}, gives isomorphisms \[\pi_1\BDiffuo(W_g\sharp\Sigma,D^{2n})\cong \pi_1\BDiffuo(W_g,D^{2n})\cong \pi_1\BDiff(W_g)\] whose composition we denote by $\psi$. We arrive at an exact sequence
\[\bfZ/2\xlra{\psi t_*}\pi_1\BDiffuo(W_g)\xlra{\psi^{-1}\iota_*}\pi_1\BDiff(W_g\sharp \Sigma)\lra0,\] whose first map agrees up to isomorphism with the collar twist of $W_g\sharp\Sigma$ on fundamental groups. Abelianising this exact sequence implies Theorem~\ref{theorem:H1} by virtue of Proposition~\ref{proposition:nontrivialitycondition}. From work of Sullivan \cite[Thm\,13.3]{Sullivan}, we know that the mapping class group $\pi_0\Diff(W_g)$ is commensurable to an arithmetic group and hence finitely generated and residually finite. By a classical result of Malcev \cite{Malcev}, this implies that it is Hopfian; that is, every surjective endomorphism is an isomorphism. Therefore, $\pi_1\BDiffuo(W_g)$ and $\pi_1\BDiffuo(W_g\sharp\Sigma)$ are isomorphic if and only if $\pi_1\SO(2n)\ra\pi_1\BDiffuo(W_g\sharp\Sigma,D^{2n})$ is trivial, and the two respective abelianisations are isomorphic if and only if this morphism is trivial after abelianisation. Thus, combining Proposition~\ref{proposition:nontrivialproducts} and \ref{proposition:trivialcollartwisting} proves Corollary~\ref{corollary:H1}.
\end{proof}

\begin{rem}\label{remark:theoremholds}The preceding proof also shows that the final parts of Theorem~\ref{theorem:H1} and Corollary~\ref{corollary:H1} are valid for $g=0$ and $g=1$ if Propositions~\ref{proposition:nontrivialitycondition} and~\ref{proposition:trivialcollartwisting} hold in these cases. This is the case for $g=1$ and most $n$, but fails for $g=0$, as discussed in Remarks~\ref{remark:remainingcases} and~\ref{remark:fails}.
\end{rem}

\begin{rem}It follows from Corollary~\ref{corollary:H1} that there are infinitely many dimensions in which the inertia group $I(W_g)$ of $W_g$ is a proper subgroup of $\Theta_{2n}$. Results of Kosinski \cite[Thm\,3.1]{Kosinski} and Wall \cite{WallInertia} show that $I(W_g)$ is in fact trivial in all dimensions.\end{rem}

\begin{proof}[Proof of Theorem~\ref{theorem:H3}]Since $\Theta_6$ vanishes by \cite{KervaireMilnor}, we may assume $2n\ge 8$. Employing the universal coefficient theorem and the fact that $\BSO(2n)$ has the homotopy type of a CW-complex of finite type, the well-known calculation of $\oH^*(\BSO(2n);\bfZ[\frac{1}{2}])$ (see e.g.\,\cite[p.\,87]{Stong}) implies that $\oH_*(\BSO(2n);\bfZ[\frac{1}{2}])$ vanishes in degrees $*\le3$ and is isomorphic to $\bfZ[\frac{1}{2}]$ in degree $4$. From \eqref{equation:evalutionsequence} and the Serre exact sequence with $\bfZ[\frac{1}{2}]$-coefficients of the homotopy fibre sequence \eqref{equation:fibresequencediffeogroups}, one sees that the morphism \[\textstyle{\oH_*(\BDiffuo(W_g\sharp\Sigma,D^{2n});\bfZ[\frac{1}{2}])\lra\oH_*(\BDiff(W_g\sharp\Sigma);\bfZ[\frac{1}{2}])}\] is an isomorphism in degrees $*\le2$ for all $\Sigma\in\Theta_{2n}$, in particular for the standard sphere. This implies the first part of the theorem, using that the groups $\oH_*(\BDiffuo(W_g\sharp\Sigma,D^{2n});\bfZ[\frac{1}{2}])$ and $\oH_*(\BDiffuo(W_g,D^{2n});\bfZ[\frac{1}{2}])$ are isomorphic by Lemma~\ref{lemma:relativediffisomorphic}. To show the second, we map the long exact sequence on homotopy groups of \eqref{equation:fibresequencediffeogroups} to the Serre exact sequence just considered. This yields a commutative diagram with exact rows
\begin{center}
\begin{tikzcd}[row sep=0.5cm, column sep=0.434cm]
\bfZ[\frac{1}{2}]\arrow[r,"t_*"]\arrow[d,"\operatorname{id}",swap]&\pi_3\BDiffuo(W_g\sharp\Sigma,D^{2n})\otimes\bfZ[\frac{1}{2}]\arrow[r]\arrow[d]&\pi_3\BDiff(W_g\sharp\Sigma)\otimes\bfZ[\frac{1}{2}]\arrow[r]\arrow[d]&0\\
\bfZ[\frac{1}{2}]\arrow[r]&\oH_3(\BDiffuo(W_g\sharp\Sigma,D^{2n});\bfZ[\frac{1}{2}])\arrow[r]&\oH_3(\BDiff(W_g\sharp\Sigma);\bfZ[\frac{1}{2}])\arrow[r]&0,
\end{tikzcd}
\end{center}
whose upper left morphism $t_*$ identifies with the induced map of the collar twist $\pi_3\SO(2n)\ra\pi_3\BDiffuo(W_g\sharp\Sigma,D^{2n})$, tensored with $\bfZ[\frac{1}{2}]$. By Corollary~\ref{corollary:collartwistingtrivial}, this morphism is trivial in the case of the standard sphere, resulting in \[\textstyle{\oH_3(\BDiffuo(W_g);\bfZ[\frac{1}{2}])\cong\oH_3(\BDiffuo(W_g,D^{2n});\bfZ[\frac{1}{2}])}.\] The latter group is in turn isomorphic to $\oH_3(\BDiffuo(W_g\sharp\Sigma,D^{2n});\bfZ[\frac{1}{2}])$ by Lemma~\ref{lemma:relativediffisomorphic}, so the lower row of the diagram provides an exact sequence as claimed. To prove the second part of the theorem, we use the stabilised Pontryagin--Thom map $\BDiffuo(W_g\sharp\Sigma,D^{2n})\ra\Omega^{\infty}_\bullet\Sigma^{-2n}\MO\langle n\rangle$ with a path component of $\Omega^{\infty}\Sigma^{-2n}\MO\langle n\rangle$ as its target (see Remark~\ref{remark:mapofspaces}) in order to extend the left square in the diagram above to a commutative diagram
\begin{center}
\begin{tikzcd}[row sep=0.5cm, column sep=0.434cm]
\bfZ[\frac{1}{2}]\arrow[r,"t_*"]\arrow[d,"\operatorname{id}",swap]&\pi_3\BDiffuo(W_g\sharp\Sigma,D^{2n})\otimes\bfZ[\frac{1}{2}]\arrow[r]\arrow[d]&\pi_3\Omega^{\infty}_\bullet\Sigma^{-2n}\MO\langle n\rangle\otimes\bfZ[\frac{1}{2}]\arrow[d]\\
\bfZ[\frac{1}{2}]\arrow[r]&\oH_3(\BDiffuo(W_g\sharp\Sigma,D^{2n});\bfZ[\frac{1}{2}])\arrow[r]&\oH_3(\Omega^{\infty}_\bullet\Sigma^{-2n}\MO\langle n\rangle;\bfZ[\frac{1}{2}]).
\end{tikzcd}
\end{center}
By the discussion leading to Proposition~\ref{proposition:nontrivialitycondition}, the upper composition is nontrivial for all homotopy spheres $\Sigma\in\Theta_{2n}$ for which the subgroup $\operatorname{im}(J)_3\cdot[\Sigma,\ell_\Sigma]\subset\pi_{2n+3}\MO\langle n\rangle$ is nontrivial away from the prime $2$. In the case $2n=10$, the group $\Theta_{10}$ is isomorphic to $\bfZ/6$, so by Proposition~\ref{proposition:nontrivialproducts}, the upper composition is nontrivial for all $10$-dimensional homotopy spheres $\Sigma$ whose order is divisible by $3$. Consequently, in order to finish the proof, it suffices to show that the Hurewicz homomorphism \[\textstyle{\pi_3\Omega^{\infty}_\bullet\Sigma^{-10}\MString\otimes\bfZ[\frac{1}{2}]\lra\oH_3(\Omega^{\infty}_\bullet\Sigma^{-10}\MString;\bfZ)\otimes\bfZ[\frac{1}{2}]}\] of the desuspended Thom spectrum $\Sigma^{-10}\MString=\Sigma^{-10}\MO\langle5\rangle$ is injective, which follows from the subsequent lemma concluding the proof.
\end{proof}

\begin{lem}The Hurewicz homomorphism
\[\pi_3\Omega^{\infty}_\bullet\Sigma^{-10}\MString\lra\oH_3(\Omega^{\infty}_\bullet\Sigma^{-10}\MString;\bfZ)\] is an isomorphism.
\end{lem}
\begin{proof}
The component $\Omega^{\infty}_\bullet\Sigma^{-10}\MString$ is the infinite loop space of the $0$-connected cover of the spectrum $\Sigma^{-10}\MString$. By \cite{Giambalvo}, its first homotopy groups are given by the table
\begin{center}
\bgroup
\def\arraystretch{1.3}
\begin{tabular}{ c | c | c | c  }
   & $i=11$ & $i=12$ & $i=13$ \\ \hline
  $\pi_i\MString$ & $0$ & $\bfZ$ & $\bfZ/3$\\
\end{tabular}
\egroup,
\end{center}
from which we see that the first possibly nontrivial $k$-invariant of the $0$-connected cover of $\Sigma^{-10}\MString$ lies in $\oH^{2}(\mathbf{HZ};\bfZ/3)$. This group is isomorphic to the unstable homology group $\oH^{5}(K(\bfZ,3);\bfZ/3)$, since more generally $\oH^k(\mathbf{H}A;B)\cong \oH^{k+n}(K(A,n);B)$ for $n>k$ and any abelian groups $A$ and $B$, which is a result of the evaluation map $\Sigma\Omega K(A,n)\ra K(A,n)$ being $(2n-1)$-connected as it is the homotopy pullback of the $(2n-1)$-connected inclusion $ K(A,n)\vee  K(A,n)\rightarrow K(A,n)\times K(A,n)$ along the diagonal. But $\oH^{5}(K(\bfZ,3);\bfZ/3)$ can be seen to vanish by employing the Serre spectral sequence of the homotopy fibre sequence $K(\bfZ,2)\ra*\ra K(\bfZ,3)$, so the $3$-truncation of the $0$-connected cover of $\Sigma^{-10}\MString$ splits into $\Sigma^{2}\mathbf{H}\bfZ$ and $\Sigma^{3}\mathbf{H}\bfZ/3$. The result is now implied by combining the induced splitting of the associated infinite loop space with the Hurewicz theorem for spaces. 
\end{proof}
\begin{rem}By Corollary~\ref{corollary:H3}, the groups \[\textstyle{\oH_{3}(\BDiff(W_g);\bfZ[\frac{1}{2}])\quad\text{and}\quad\oH_{3}(\BDiff(W_g\sharp\Sigma);\bfZ[\frac{1}{2}])}\] are not isomorphic for $g\ge0$ and any homotopy sphere $\Sigma\in\Theta_{10}$ whose order is divisible by $3$. In fact, in this dimension, one can explicitly calculate the homology of $\BDiff(W_g\sharp\Sigma)$ in low degrees for large $g$, at least after inverting $2$. One starts by computing the groups $\oH_*(\BDiffuo(W_g,D^{10});\bfZ[\frac{1}{2}])$ in low degrees for large $g$ by combining computations of the homotopy groups of $\MString$ \cite{Giambalvo,HoveyRavenel} with a calculation of the first $k$-invariants of the $0$-connected cover of $\MTString(10)$ and an extension of a method of Galatius--Randal-Williams \cite[Ch.\,5]{GRWabelian}. From this, the groups in question can be obtained using the ideas of the proof of Theorem~\ref{theorem:H3}. The following table describes the $i$th homology with $\bfZ[\frac{1}{2}]$-coefficients of $\BDiff(W_g\sharp\Sigma)$ for $g\ge{2i+3}$. The first row displays the respective homology groups for the two homotopy spheres $\Sigma\in\Theta_{10}\cong\bfZ/6$ of order $0$ and $2$, whereas the second row shows the respective groups for the other four homotopy spheres in $\Theta_{10}$.
\begin{center}
\bgroup
\def\arraystretch{1.3}
\begin{tabular}{ c | c | c | c  }
   & $i=1$ & $i=2$ & $i=3$ \\ \hline
  $\oH_i(\BDiff(W_g);\bfZ[\frac{1}{2}])$ & $0$ & $\bfZ[\frac{1}{2}]$ & $(\bfZ/3)^2$\\\hline
   $\oH_i(\BDiff(W_g\sharp\Sigma);\bfZ[\frac{1}{2}])$ & $0$ & $\bfZ[\frac{1}{2}]$ & $\bfZ/3$\\ 
\end{tabular}
\egroup
\end{center}
\end{rem}

\bibliographystyle{amsalpha}
\bibliography{literature}
\end{document}